\documentclass[11pt]{article}
\usepackage{mathrsfs}
\usepackage{amsmath}
\usepackage{amsfonts}
\usepackage{amssymb}
\usepackage{latexsym}
\usepackage{enumerate}
\usepackage{graphicx}
\usepackage{color}
\usepackage{hyperref}
\newtheorem{theorem}{\bf Theorem}[section]
\newtheorem{lemma}[theorem]{\bf Lemma}

\newtheorem{corollary}[theorem]{\bf Corollary}

\newenvironment{proof}{\noindent{\em Proof:}}{\quad \hfill$\Box$\vspace{2ex}}

\def \bN {\mathbb N}
\def \bZ {\mathbb Z}
\def \bE {\mathbb E}

\def \bR {\mathbb R}

\def \bE {\mathbb E}

\def \cB {{\cal B}}

\def \and {\, \mbox{\rm and}\, }
\def \sinc {\,{\rm sinc}\,}

\def \supp {\,{\rm supp}\,}

\makeatletter

\newcommand{\Rmnum}[1]{\expandafter\@slowromancap\romannumeral #1@}
\makeatother

\begin{document}
\title{\bf An Optimal Convergence Rate for the Gaussian Regularized Shannon Sampling Series}
\author{Rongrong Lin\thanks{School of Data and Computer Science, Sun Yat-sen University, Guangzhou 510006, P. R. China. E-mail address: linrr@mail2.sysu.edu.cn.}}
\date{ }
\maketitle

\begin{abstract} We consider the reconstruction of a bandlimited function from its finite localized sample data. Truncating the classical Shannon sampling series results in an unsatisfactory convergence rate due to the slow decay of the sinc function. To overcome this drawback, a simple and highly effective method, called the Gaussian regularization of the Shannon series, was proposed in the engineering and has received remarkable attention. It works by multiplying the sinc function in the Shannon series with a regularized Gaussian function. Recently, it was proved that the upper error bound of this method can achieve a convergence rate of the order $O(\frac{1}{\sqrt{n}}\exp(-\frac{\pi-\delta}{2}n))$, where $0<\delta<\pi$ is the bandwidth and $n$ the number of sample data.  The convergence rate is by far the best convergence rate among all regularized methods for the Shannon sampling series. The main objective of this article is to present the theoretical justification and numerical verification that the convergence rate is optimal when $0<\delta<\pi/2$ by estimating the lower error bound of the truncated Gaussian regularized Shannon sampling series.


\end{abstract}

\section{Introduction}

The classical Shannon sampling theorem \cite{Jerri, Kotelnikov33,Nyquist28, Shannon,Unser,Whittaker} states that any bandlimited function with bandwidth $\pi$ can be completely reconstructed by its infinite samples at integers. In practice, however, we can only sum over finite sample data ``near'' $t$ to approximate the function value at $t$.
Truncating the classical Shannon sampling series \cite{Helms,Jagerman} results in a convergence rate of the order $O(\frac{1}{\sqrt{n}})$ due to the slow decay of the sinc function, where $n$ denotes the number of samples. Moreover, this convergence rate for the truncated Shannon series is optimal in the worst case scenario (see, e.g., Lemma 1.1, \cite{Micchelli}).
A useful way to significantly improve the convergence rate is by oversampling. Led by this idea, three regularization methods \cite{Jagerman,Micchelli,Qian} for Shannon's sampling series were proposed to reconstruct a bandlimited function $f$ with bandwidth $0<\delta<\pi$ from its finite oversampling data $\{f(j):j=-n+1,-n+2,\dots,n\}$ with an exponentially decaying approximation error.  They work by multiplying the sinc function  in the Shannon series with a rapidly-decaying regularization function, namely a power of a sinc function \cite{Jagerman}, a spline function \cite{Micchelli}, or a Gaussian function \cite{LinZhang17,Qian,Qianthesis}.  To be precise and to state the purpose of the paper, the {\it truncated Gaussian regularized Shannon sampling series} proposed by Wei \cite{Wei98} is defined as
\begin{equation}\label{GR}
(S_{n,r}f)(t):=\sum_{j=-n+1}^n f(j)\sinc(t-j)e^{-\frac{(t-j)^2}{2r^2}},\ t\in(0,1),\ f\in\cB_\delta, 
\end{equation}
where $r>0$, $n\ge2$, $0<\delta<\pi$, and $\sinc(x):=\sin(\pi x)/(\pi x)$, $x\in\bR$.  The {\it Paley-Wiener space} $\cB_\delta$ is defined as
$$
\cB_\delta:=\Big\{f\in L^2(\bR)\cap C(\bR): \supp\hat{f}\subseteq [-\delta,\delta]\Big\}
$$
with the norm $\|f\|_{\cB_\delta}:=\|f\|_{L^2(\bR)}=\|\hat{f}\|_{L^2(\bR)}$. Denote by $\supp h$ the support of a function $h$.
In this note, for each $f\in L^1(\bR)$, its {\it Fourier transform} \cite{Debnath15} takes the form:
$$
\hat{f}(\xi):=\frac{1}{\sqrt{2\pi}} \int_{\bR} f(x)e^{-ix\xi }dx,\ \xi\in\bR.
$$
We can extend the Fourier transform to $L^2(\bR)$ by the standard approximation process.

Given a bandlimited function with bandwidth $0<\delta<\pi$, it was proved in \cite{Qian} that  the upper error bound of the truncated Gaussian regularized Shannon sampling series (\ref{GR}) is of the convergence rate of the order $O(\sqrt{n}\exp(-\frac{\pi-\delta}2n))$ after optimizing about the variance $r$ of the Gaussian function as done in (\cite{Micchelli},  pp. 106-107). 
Recently, the paper \cite{LinZhang17} provided a better estimate for the second error term, namely $E_2f$ in (\ref{E1E2eq}) or Equation (2.5) in \cite{Qian},  and hence improved  the convergence rate for (\ref{GR}) to the order
\begin{equation}\label{rate}
O\Big(\frac{1}{\sqrt{n}}\exp(-\frac{\pi-\delta}{2}n)\Big).
\end{equation}
The latter rate (\ref{rate}) is by far the best convergence rate among all regularized methods \cite{Jagerman,LinZhang17,Micchelli,Qian} for the Shannon sampling series. Due to its simplicity and high accuracy, the Gaussian regularized Shannon sampling series has been widely applied to scientific and engineering computations.  We notice that more than a hundred such papers have appeared (see \url{http://www.math.msu.edu/~wei/pub-sec.html} for the list, and \cite{Wei98,Wei00,Wei000,Wei08,Zhao} for comments and discussion). Furthermore, many generalizations of Shannon's sampling theorem have been established (see, e.g., \cite{Aldroubi,Chen15,Chen1,Grochenig,Marks,SongSun07,SunZhou02,Vaidyanathan} and references therein).

Numerical experiments in \cite{LinZhang17} indicated that the convergence rate (\ref{rate}) should be optimal. 
We mathematically certify this indication by estimating the lower error bound of (\ref{GR}). In addition, we are able to justify that the selection of the variance $r=\sqrt{\frac{n-1}{\pi-\delta}}$ in (\ref{GR}) is optimal. We emphasize that it is the first lower error bound estimation for (\ref{GR}) in the literature. More importantly, the lower bound is of exponential decay. 

For any $x>0$, we denote by $\lceil x\rceil$ the smallest integer greater than or equal to $x$. We now are ready to present our main results. 
\begin{theorem}\label{Theorem} Let $0<\delta<\frac{\pi}{2}$ and  $0<\varepsilon<1$. If there exists $\frac{2}{\sqrt{\varepsilon(2+\varepsilon)}(\pi-\delta)}\le r\le \sqrt{\frac{n-1}{\pi-\delta}}$ such that $C_{r,\delta,\varepsilon}>0$, then the lower error bound for the truncated Gaussian regularized Shannon sampling series (\ref{GR}) is
$$
\sup_{\|f\|_{\cB_\delta}\le 1}\sup_{t\in(0,1)}\big|f(t)-(S_{n,r}f)(t)\big|
 \ge  \frac{1}{\pi\sqrt{2\delta}} \Big[C_{r,\delta,\varepsilon} \frac{e^{-\frac{(\pi-\delta)^2r^2}{2}}}{r^3}-\frac{2\sqrt{2} r^2 e^{-\frac{(n-1)^2}{2r^2}}}{n(n+\frac12) (n-1)\sqrt{\pi}}\Big],
 $$
where $n\ge \lceil \frac{4}{\varepsilon(2+\varepsilon)(\pi-\delta)}+1\rceil$ is sufficiently large and
\begin{equation}\label{Crdelta}
C_{r,\delta,\varepsilon}:=\sin\Big(\frac{\delta}{2}\Big)\Big[\frac{4}{(2+\varepsilon)(\pi+\delta) \delta}\Big(\frac{2}{(2+\varepsilon)(\pi-\delta)}-\frac{e^{-2\pi\delta r^2}}{\pi+\delta}\Big) 
-\frac{2}{(2\pi-\delta)(\pi-\delta)^2} \Big].
\end{equation}
\end{theorem} 

We should make some comments for Theorem \ref{Theorem} here. To begin with, one sees that $e^{-\frac{(\pi-\delta)^2r^2}{2}}$ is decreasing on $(0,\infty)$ with respect to the variable $r$.  The condition $0<r\le \sqrt{\frac{n-1}{\pi-\delta}}$  guarantees  $e^{-\frac{(\pi-\delta)^2r^2}{2}}\ge e^{-\frac{(n-1)^2}{2r^2}}$. Second, the parameter $0<\varepsilon<1$  can be arbitrarily small and will be formally introduced in (\ref{Millseq2}). 
Third, an easy computation shows that  $\sqrt{\frac{n-1}{\pi-\delta}}\ge \frac{2}{\sqrt{\varepsilon(2+\varepsilon)}(\pi-\delta)}$ implies $n\ge \frac{4}{\varepsilon(2+\varepsilon)(\pi-\delta)}+1$. Finally,  $0<\delta<\frac{\pi}{2}$ in Theorem \ref{Theorem} is a necessary condition such that $C_{r,\delta,\varepsilon}>0$. It is straightforward to see that
$$
0<\frac{4}{(2+\varepsilon)(\pi+\delta) \delta}\frac{2}{(2+\varepsilon)(\pi-\delta)}-\frac{2}{(2\pi-\delta)(\pi-\delta)^2}<\frac{1}{(\pi+\delta) \delta}\frac{2}{\pi-\delta}-\frac{2}{(2\pi-\delta)(\pi-\delta)^2},
$$
namely, $\frac{1}{(\pi+\delta)\delta}>\frac{1}{(2\pi-\delta)(\pi-\delta)}$, which implies $\delta<\frac{\pi}{2}$. 
 
It was proved in \cite{LinZhang17,Micchelli} that the best convergence rate (\ref{rate}) for the upper error bound of (\ref{GR}) can be obtained when $r=\sqrt{\frac{n-1}{\pi-\delta}}$. 
To facilitate a comparison with (\ref{rate}),  we obtain a lower error bound of the order $O\big(\frac{1}{n\sqrt{n}}\exp(-\frac{\pi-\delta}{2}n)\big)$.
\begin{corollary}\label{Corollary} Let $r:=\sqrt{\frac{n-1}{\pi-\delta}}$ in Theorem \ref{Theorem} and $0<\varepsilon<1$.   If  $C_{r,\delta,\varepsilon}>0$, then
\begin{equation}\label{remarkeq1} 
\sup_{\|f\|_{\cB_\delta}\le 1}\sup_{t\in(0,1)}\big|f(t)-(S_{n,r}f)(t)\big|\ge \Big[C_{r,\delta,\varepsilon}(\pi-\delta)\sqrt{\pi-\delta}-\frac{2\sqrt{2}}{(\pi-\delta)\sqrt{\pi}} \frac{(n-1)\sqrt{n-1}}{n(n+\frac12)}\Big]
\frac{e^{-\frac{(\pi-\delta)(n-1)}{2}}}{\pi\sqrt{2\delta}(n-1)\sqrt{n-1}},
\end{equation}
where
\begin{equation}\label{n}
n\ge \max\Big\{2, \Big\lceil \frac{4}{\varepsilon(2+\varepsilon)(\pi-\delta)}+1\Big\rceil, \Big\lceil \frac{8}{\pi C_{r,\delta,\varepsilon}^2(\pi-\delta)^5}-1 \Big\rceil \Big\}.
\end{equation}
\end{corollary}

Some remarks should be made for Corollary \ref{Corollary}.
First of all, we have $\frac{(n-1)\sqrt{n-1}}{n(n+\frac12)}<\frac{1}{\sqrt{n+1}}$ for all $n\ge2$. As a result, the condition $n\ge  \frac{8}{\pi C_{r,\delta,\varepsilon}^2(\pi-\delta)^5}-1$ guarantees
$$
C_{r,\delta,\varepsilon}(\pi-\delta)\sqrt{\pi-\delta}-\frac{2\sqrt{2}}{(\pi-\delta)\sqrt{\pi}} \frac{(n-1)\sqrt{n-1}}{n(n+\frac12)}>C_{r,\delta,\varepsilon}(\pi-\delta)\sqrt{\pi-\delta}-\frac{2\sqrt{2}}{(\pi-\delta)\sqrt{\pi}}\frac{1}{\sqrt{n+1}}\ge0.
$$
Secondly, note that $\sqrt{\frac{n-1}{\pi-\delta}}\ge \frac{1}{\sqrt{\pi-\delta}}$ for all $n\ge2$. Clearly, a {\it sufficient condition} for $\delta$ such that $C_{r,\delta,\varepsilon}>0$ is
$$
\frac{4}{(2+\varepsilon)(\pi+\delta) \delta}\Big(\frac{2}{(2+\varepsilon)(\pi-\delta)}-\frac{e^{-\frac{2\pi\delta}{\pi-\delta}}}{\pi+\delta}\Big)-\frac{2}{(2\pi-\delta)(\pi-\delta)^2}>0.
$$
Numerical result shows that $\frac{1}{200}\pi\le \delta\le \frac{49}{100}\pi$ is a sufficient condition such that $C_{r,\delta,\varepsilon}>0$ when $\varepsilon=1/20$, where $C_{r,\delta,\varepsilon}$ is defined by (\ref{Crdelta}). Finally, we should mention that $C_{r,\delta,\varepsilon}<0$ as the bandwidth $\delta\to0$ or $\delta\to\frac{\pi}{2}$.

Now, we are in a position to present the upper error bound for (\ref{GR}). It is worth pointing out that an additional error term was lost in \cite{Qian}.  
The error term $\frac{1}{\sqrt{2\pi}}\|\bE_1\hat{f}\|_{L^1(\bR)}$ instead of $\frac{1}{\sqrt{2\pi}}\|\widehat{E_1f}\|_{L^1(\bR)}$ was actually estimated in (\cite{Qian}, p. 1173), where $\widehat{E_1f}=\bE_1\hat{f}+\bE_2\hat{f}$ (see (\ref{E1E2eq0}) and (\ref{E1bE1bE2})). Hence, the problem amounts to estimating the additional term $\frac{1}{\sqrt{2\pi}}\|\bE_2\hat{f}\|_{L^1(\bR)}$ (see (\ref{upperboundeq1})). 
We will see that the missing term $\bE_2 \hat{f}$ has the same convergence rate as $\bE_1 \hat{f}$ (see (\ref{upperboundeq1}) and (\ref{upperboundeq2})). 
As a result, the new convergence rate for the first error term $E_1f$ in (\ref{E1E2eq}) remains the same as $\bE_1\hat{f}$. 

Precisely speaking, we actually obtain the following upper error bound estimation. 
\begin{theorem}\label{Theorem2} Let $0<\delta<\pi$, $r>0$ and $n\ge2$. Then the upper error bound for the truncated Gaussian regularized Shannon sampling series (\ref{GR}) is
$$
\sup_{\|f\|_{\cB_\delta}\le 1}\sup_{t\in(0,1)}\big|f(t)-(S_{n,r}f)(t)\big|\le \Big[\frac{1+\big(1+\frac{1}{2\pi(3\pi-\delta)r^2}\big)e^{-2\pi(2\pi-\delta)r^2}}{\sqrt{2(\pi-\delta)}r}+\sqrt{2\delta}\Big]\frac{e^{-\frac{(\pi-\delta)^2r^2}{2}}}{\pi (\pi-\delta)r} +\frac{r e^{-\frac{(n-1)^2}{2r^2}}}{\pi n\sqrt{n-1}}.
$$
\end{theorem}

We should again discuss the special case: $r:=\sqrt{\frac{n-1}{\pi-\delta}}$. In this case, note that $r^2=\frac{n-1}{\pi-\delta}\ge \frac{1}{\pi-\delta}$ for all $n\ge2$ and
$$
1+\Big(1+\frac{1}{2\pi(3\pi-\delta)r^2}\Big)e^{-2\pi(2\pi-\delta)r^2}\le 1+\Big(1+\frac{\pi-\delta}{2\pi(3\pi-\delta)}\Big)e^{-\frac{2\pi(2\pi-\delta)}{\pi-\delta}}\le  1+(1+\frac{1}{6\pi})e^{-4\pi}.
$$
\begin{corollary}  Let $r=\sqrt{\frac{n-1}{\pi-\delta}}$ in Theorem \ref{Theorem2}. Then
\begin{equation}\label{remarkeq2}
\sup_{\|f\|_{\cB_\delta}\le 1}\sup_{t\in(0,1)}\big|f(t)-(S_{n,r}f)(t)\big|\le\Big(\sqrt{2\delta}+\frac{\sqrt{n-1}}{n}+\frac{1+(1+\frac{1}{6\pi})e^{-4\pi}}{\sqrt{2(n-1)}}\Big)\frac{e^{-\frac{(\pi-\delta)(n-1)}{2}}}{\pi\sqrt{(\pi-\delta)(n-1)}}.
\end{equation}
\end{corollary}

By Theorems \ref{Theorem} and \ref{Theorem2}, one sees that the convergence rate of the order $O\big(\frac{1}{r}\exp(-\frac{(\pi-\delta)^2r^2}{2})\big)$ for the truncated Gaussian regularized Shannon sampling series is {\it optimal} when $\frac{2}{\sqrt{\varepsilon(2+\varepsilon)}(\pi-\delta)}\le r\le \sqrt{\frac{n-1}{\pi-\delta}}$ and $C_{r,\delta,\varepsilon}>0$. Recall that $e^{-\frac{(\pi-\delta)^2r^2}{2}}$ is decreasing on $(0,\infty)$. Thus, when $r= \sqrt{\frac{n-1}{\pi-\delta}}$ and $C_{r,\delta,\varepsilon}>0$, by (\ref{remarkeq1}) and (\ref{remarkeq2}), we conclude that the convergence rate (\ref{rate}) is optimal. In other words,  it turns out that the selection $r=\sqrt{\frac{n-1}{\pi-\delta}}$ is optimal in Theorem \ref{Theorem2} among all $\frac{2}{\sqrt{\varepsilon(2+\varepsilon)}(\pi-\delta)}\le r\le \sqrt{\frac{n-1}{\pi-\delta}}$, where $0<\varepsilon<1$ and $n$ is given by (\ref{n}).

The paper is organized as follows.
We shall exploit Fourier analysis techniques to prove Theorems \ref{Theorem} and \ref{Theorem2} in the following section. In the last section, numerical experiments are presented to demonstrate our theoretical results. Specifically, the lower error bound, the reconstruction error, and the upper error bound of the truncated Gaussian regularized Shannon sampling series (\ref{GR}) will be directly compared.

\section{Proof of Theorems \ref{Theorem} and \ref{Theorem2}}

In this section, we are devoted to proving Theorems \ref{Theorem} and \ref{Theorem2}. Assume $0<\delta<\pi$, $0<\varepsilon<1$, and $r>0$. For any subset $E$ of $\bR$, denote by ${\bf 1}_{E}$ the characteristic function on $E$, namely,  ${\bf 1}_{E}(x)=1$ if $x\in E$ and $0$ otherwise. Denote by $\bZ$ the set of all integers and $\bN$ the set of all positive integers.

We first deal with Theorem \ref{Theorem}. To do so, we divide the error $f-S_{n,r}f$ into two terms. Specifically, 
for every $f\in\cB_\delta$, $0<\delta<\pi$, we set
\begin{equation}\label{E1E2eq} 
f(t)-(S_{n,r}f)(t):=(E_1 f)(t)+(E_2f)(t),\ t\in(0,1),
\end{equation}
where the first error term $E_1f$ and the second error term $E_2f$ are defined by
\begin{equation}\label{E1E2eq0} 
(E_1f)(t):=f(t)-\sum_{j\in\bZ} f(j)\sinc(t-j)e^{-\frac{(t-j)^2}{2r^2}}\mbox{ and } (E_2f)(t):=\sum_{j\notin(-n,n]} f(j)\sinc(t-j)e^{-\frac{(t-j)^2}{2r^2}},
\end{equation}
respectively. It follows that
\begin{equation}\label{E1E2} 
\sup_{\|f\|_{\cB_\delta}\le 1}\sup_{t\in(0,1)}\big|f(t)-(S_{n,r}f)(t)\big|\ge \big|(E_1 f_0)(t_0)+(E_2f_0)(t_0)\big|\ge |(E_1 f_0)(t_0)|-|(E_2f_0)(t_0)|,
\end{equation}
where $f_0\in\cB_\delta$ with $\|f_0\|_{\cB_\delta}\le 1$ and the point $t_0\in(0,1)$  is appropriately chosen. Throughout this note, we always take
\begin{equation}\label{f0}
f_0(t):=\frac{1}{\sqrt{\pi\delta}}\frac{\sin((t-\frac12)\delta)}{(t-\frac12)}, \ t\in\bR,
\end{equation}
and $t_0\in (0,1)$ in (\ref{E1E2}) such that the function value of the function $|E_1 f_0|$  at $t_0$ is not less than its average on $(0,1)$, that is,  
\begin{equation}\label{t0} 
|(E_1 f_0)(t_0)|\ge \|(E_1 f_0)(t)\|_{L^1((0,1))}.
\end{equation}
Clearly, $f_0\in\cB_\delta$, $\|f_0\|_{\cB_\delta}=1$, and  its Fourier transform is $\widehat{f_0}(\xi)=\frac{1}{\sqrt{2\delta}}{\bf 1}_{[-\delta,\delta]}(\xi) e^{-i \xi/2}$, $\xi\in\bR$.

Therefore, our task reduces to give a ``big'' lower bound for $|(E_1 f_0)(t_0)|$ and  a ``small'' upper bound for $|(E_2 f_0)(t_0)|$, where $|(E_1 f_0)(t_0)|$ and $|(E_2 f_0)(t_0)|$ are defined by (\ref{E1E2}).

For the sake of clarity, two lemmas are needed. 
\begin{lemma}\label{Lemma1} Let $0<\delta<\pi$  and $0<\varepsilon<1$.  If there exists $r\ge \frac{2}{\sqrt{\varepsilon(2+\varepsilon)}(\pi-\delta)}$ such that $\frac{2}{(2+\varepsilon)(\pi-\delta)}-\frac{e^{-2\pi\delta r^2}}{\pi+\delta}>0$, then 
$$
\int_{-\delta}^{\delta} \Big(1-\frac{1}{\sqrt{\pi}}\int_{\frac{(\xi-\pi)r}{\sqrt{2}}}^{\frac{(\xi+\pi)r}{\sqrt{2}}} e^{-\tau^2}d\tau\Big)d\xi
>\frac{2\sqrt{2}}{(2+\varepsilon)(\pi+\delta) \sqrt{\pi}}\Big(\frac{2}{(2+\varepsilon)(\pi-\delta)}-\frac{e^{-2\pi\delta r^2}}{\pi+\delta}\Big)\frac{e^{-\frac{(\pi-\delta)^2r^2}{2}}}{r^3}.
$$
\end{lemma}
\begin{proof} By Mills' ratio \cite{Pollak}
\begin{equation}\label{Millseq}
\frac{e^{-x^2}}{x+\sqrt{x^2+2}}\le \int_x^{\infty} e^{-\tau^2}d\tau\le \frac{e^{-x^2}}{x+\sqrt{x^2+\frac{4}{\pi}}},\ x>0,
\end{equation}
we obtain
\begin{equation}\label{Millseq1}
\int_x^{\infty} e^{-\tau^2}d\tau<\frac{e^{-x^2}}{2x},\ x>0.
\end{equation}
 By (\ref{Millseq}), for any $0<\varepsilon<1$, we have
\begin{equation}\label{Millseq2}
\frac{e^{-x^2}}{(2+\varepsilon)x}< \int_x^{\infty} e^{-\tau^2}d\tau\mbox{ for all }x\ge\sqrt{\frac{2}{\varepsilon(2+\varepsilon)}}.
\end{equation}
Clearly,  the parameter $\varepsilon$ serves to improve the accuracy of the inequality (\ref{Millseq2}).
Observe that
$$
\begin{array}{ll}
\displaystyle{1-\frac{1}{\sqrt{\pi}}\int_{\frac{(\xi-\pi)r}{\sqrt{2}}}^{\frac{(\xi+\pi)r}{\sqrt{2}}} e^{-\tau^2}d\tau }
&\displaystyle{=\frac{1}{\sqrt{\pi}}\int_{\bR} e^{-\tau^2}d\tau-\frac{1}{\sqrt{\pi}}\int_{\frac{(\xi-\pi)r}{\sqrt{2}}}^{\frac{(\xi+\pi)r}{\sqrt{2}}} e^{-\tau^2}d\tau }\\
&\displaystyle{= \frac{1}{\sqrt{\pi}} \Big[\int_{\frac{(\xi+\pi)r}{\sqrt{2}}}^{\infty} e^{-\tau^2}d\tau + \int^{\frac{(\xi-\pi)r}{\sqrt{2}}}_{-\infty} e^{-\tau^2}d\tau \Big]}\\
&\displaystyle{= \frac{1}{\sqrt{\pi}} \Big[\int_{\frac{(\xi+\pi)r}{\sqrt{2}}}^{\infty} e^{-\tau^2}d\tau + \int_{\frac{(\pi-\xi)r}{\sqrt{2}}}^{\infty} e^{-\tau^2}d\tau \Big].}
\end{array}
$$
By (\ref{Millseq2}), for any $\xi\in[-\delta,\delta]$  and $r\ge \frac{2}{\sqrt{\varepsilon(2+\varepsilon)}(\pi-\delta)}$ (i.e., $\frac{(\pi-\delta)r}{\sqrt{2}}\ge\sqrt{\frac{2}{\varepsilon(2+\varepsilon)}}$), we have
$$
1-\frac{1}{\sqrt{\pi}}\int_{\frac{(\xi-\pi)r}{\sqrt{2}}}^{\frac{(\xi+\pi)r}{\sqrt{2}}} e^{-\tau^2}d\tau 
> \frac{1}{(2+\varepsilon)\sqrt{\pi}} \Big[\frac{e^{-\frac{(\pi+\xi)^2r^2}{2}}}{(\pi+\xi)r/\sqrt{2}}+\frac{e^{-\frac{(\pi-\xi)^2r^2}{2}}}{(\pi-\xi)r/\sqrt{2}}\Big]
:=\frac{\sqrt{2}}{(2+\varepsilon) r\sqrt{\pi}}G(\xi),
$$
where
$$
G(\xi):=\frac{e^{-\frac{(\pi+\xi)^2r^2}{2}}}{\pi+\xi}+\frac{e^{-\frac{(\pi-\xi)^2r^2}{2}}}{\pi-\xi},\ \xi\in[-\delta,\delta].
$$
We compute
$$
\begin{array}{ll}
\displaystyle{ \int_{-\delta}^{\delta} G(\xi)d\xi=2\int_{-\delta}^{\delta}\frac{e^{-\frac{(\pi+\xi)^2r^2}{2}}}{\pi+\xi}d\xi }
&\displaystyle{ >\frac{2}{\pi+\delta}\int_{-\delta}^{\delta}e^{-\frac{(\pi+\xi)^2r^2}{2}}d\xi }\\
&\displaystyle{=\frac{2\sqrt{2}}{(\pi+\delta)r}\int_{\frac{(\pi-\delta)r}{\sqrt{2}}}^{\frac{(\pi+\delta)r}{\sqrt{2}}} e^{-\tau^2}d\tau }\\
&\displaystyle{=\frac{2\sqrt{2}}{(\pi+\delta)r}\Big[\int_{\frac{(\pi-\delta)r}{\sqrt{2}}}^{\infty} e^{-\tau^2}d\tau-\int_{\frac{(\pi+\delta)r}{\sqrt{2}}}^{\infty} e^{-\tau^2}d\tau\Big]. }\\
\end{array}
$$
Combining (\ref{Millseq1}) with (\ref{Millseq2}), we obtain
$$
\begin{array}{ll}
\displaystyle{\int_{-\delta}^{\delta} G(\xi)d\xi}
&\displaystyle{>\frac{2\sqrt{2}}{(\pi+\delta)r}\Big(\frac{\sqrt{2}}{(2+\varepsilon)(\pi-\delta)r}e^{-\frac{(\pi-\delta)^2r^2}{2}}-\frac{1}{\sqrt{2}(\pi+\delta)r}e^{-\frac{(\pi+\delta)^2r^2}{2}}\Big) }\\
&\displaystyle{=\frac{2}{(\pi+\delta)r^2}\Big(\frac{2}{(2+\varepsilon)(\pi-\delta)}-\frac{e^{-2\pi\delta r^2}}{\pi+\delta}\Big)e^{-\frac{(\pi-\delta)^2r^2}{2}}.}\\
\end{array}
$$
The proof is hence complete.
\end{proof}

\begin{lemma}\label{lemma2} Let $0<\delta<\pi$ and $r>0$. It holds
$$
\sum_{k\in\bN}\int_{-\delta+2k\pi}^{\delta+2k\pi} e^{-\frac{(\xi-\pi)^2r^2}{2}}  d\xi<\frac{e^{-\frac{(\pi-\delta)^2r^2}{2}}}{(\pi-\delta)r^2}.
$$
\end{lemma}
\begin{proof} By a change of variables $\tau=\frac{(\xi-\pi)r}{\sqrt{2}}$, we have by (\ref{Millseq1})
$$
\sum_{k\in\bN}\int_{-\delta+2k\pi}^{\delta+2k\pi} e^{-\frac{(\xi-\pi)^2r^2}{2}}  d\xi 
=\frac{\sqrt{2}}{r}\sum_{k\in\bN}\int_{\frac{[(2k-1)\pi-\delta]r}{\sqrt{2}}}^{\frac{[(2k-1)\pi+\delta]r}{\sqrt{2}}}    e^{-\tau^2}  d\tau 
<\frac{\sqrt{2}}{r}\int_{\frac{(\pi-\delta)r}{\sqrt{2}}}^{\infty}    e^{-\tau^2}  d\tau 
<\frac{e^{-\frac{(\pi-\delta)^2r^2}{2}}}{(\pi-\delta)r^2},
$$
which completes the proof.
\end{proof}

The convolution of two functions $f,g$ on $\bR$ is given as
$$
(f*g)(x):=\frac{1}{\sqrt{2\pi}}\int_{\bR} f(x-t)g(t)dt,\ x\in\bR,
$$
whenever the integral is well-defined (see, e.g., \cite{Debnath15}, pp. 39--41). 

Now, we are ready to estimate the lower bound of $|(E_1f_0)(t_0)|$ mentioned at the beginning of this section.
\begin{theorem}\label{Theorem1} Let $0<\delta<\frac{\pi}{2}$ and $0<\varepsilon<1$.  If there exists $r\ge \frac{2}{\sqrt{\varepsilon(2+\varepsilon)}(\pi-\delta)}$ such that $C_{r,\delta,\varepsilon}>0$, then
$$
|(E_1f_0)(t_0)|\ge \frac{C_{r,\delta,\varepsilon}}{\pi\sqrt{2\delta}} \frac{e^{-\frac{(\pi-\delta)^2r^2}{2}}}{r^3},
$$
where $C_{r,\delta,\varepsilon}$ is given by (\ref{Crdelta}).
\end{theorem}
\begin{proof}  Let $f\in\cB_\delta$, $0<\delta<\pi$. Then 
\begin{equation}\label{hatf}
\hat{f}(\xi)=\mathring{\hat{f}}(\xi):=\frac{1}{\sqrt{2\pi}}\sum_{j\in\bZ}f(j)e^{-ij\xi},\ \xi\in[-\delta,\delta], 
\end{equation}
where $f(j)=\frac{1}{\sqrt{2\pi}}\int_{-\pi}^{\pi} \hat{f}(\xi)e^{ij\xi}d\xi$, $j\in\bZ$. Note that $\widehat{g\cdot h}=\hat{g}*\hat{h}$ for any $g,h\in L^2(\bR)$, and for any $a>0$, $\widehat{e^{-a x^2}}(\xi)=\frac{1}{\sqrt{2a}}e^{-\frac{\xi^2}{4a}}$ and $\widehat{{\bf 1}_{[-a,a]}}(\xi)=\sqrt{\frac{2}{\pi}}\frac{\sin(a\xi)}{\xi}$, $\xi\in\bR$. By (\ref{E1E2eq0}), we have
$$
\begin{array}{ll}
 \displaystyle{\widehat{E_1f}(\xi)  }
& \displaystyle{ =\hat{f}(\xi)-\sum_{j\in\bZ}f(j)\Big[\Big(\sinc(\cdot)e^{-\frac{\cdot^2}{2r^2}}\Big)(t-j)\Big]^{\hat{\,}}(\xi) }\\  
& \displaystyle{=\hat{f}(\xi)-\sum_{j\in\bZ}f(j)e^{-ij\xi} \Big(\sinc(\cdot)e^{-\frac{\cdot^2}{2r^2}}\Big)^{\hat{\,}}(\xi)  } \\
& \displaystyle{=\hat{f}(\xi)-\sum_{j\in\bZ}f(j)e^{-ij\xi} \Big(\widehat{\sinc(\cdot)} * \widehat{e^{-\frac{\cdot^2}{2r^2}}}\Big)(\xi) } \\
& \displaystyle{=\hat{f}(\xi)-\sum_{j\in\bZ}f(j)e^{-ij\xi} \Big(\frac{1}{\sqrt{2\pi}}{\bf 1}_{[-\pi,\pi]}(\xi)* re^{-\frac{r^2\xi^2}{2}} \Big)}\\
& \displaystyle{=\hat{f}(\xi)-\sum_{j\in\bZ}f(j)e^{-ij\xi}\frac{1}{2\pi}\int_{\xi-\pi}^{\xi+\pi} re^{-\frac{r^2\eta^2}{2}}d \eta,\ \xi\in\bR.}
\end{array}
$$
Note that $\supp\hat{f}\subseteq[-\delta,\delta]$, and $\mathring{\hat{f}}$ in (\ref{hatf}) is a $2\pi$-periodic function on $\bR$ with $\supp \mathring{\hat{f}}=\cup_{k\in\bZ}[-\delta+2k\pi,\delta+2k\pi]$. By (\ref{hatf}), we obtain
\begin{equation}\label{E1bE1bE2}
\begin{array}{ll}
\widehat{E_1f}(\xi)
&\displaystyle{=\hat{f}(\xi)-\mathring{\hat{f}}(\xi)\frac{1}{\sqrt{2\pi}}\int_{\xi-\pi}^{\xi+\pi} re^{-\frac{r^2\eta^2}{2}}d \eta}\\
&\displaystyle{=\hat{f}(\xi)-\mathring{\hat{f}}(\xi)\frac{1}{\sqrt{\pi}}\int_{\frac{(\xi-\pi)r}{\sqrt{2}}}^{\frac{(\xi+\pi)r}{\sqrt{2}}} e^{-\tau^2}d\tau}\\
&\displaystyle{=\hat{f}(\xi)\Big(1-\frac{1}{\sqrt{\pi}}\int_{\frac{(\xi-\pi)r}{\sqrt{2}}}^{\frac{(\xi+\pi)r}{\sqrt{2}}} e^{-\tau^2}d\tau\Big)
+\mathring{\hat{f}}(\xi){\bf 1}_{\bR\setminus[-\pi,\pi]}(\xi)\frac{1}{\sqrt{\pi}}\int_{\frac{(\xi-\pi)r}{\sqrt{2}}}^{\frac{(\xi+\pi)r}{\sqrt{2}}} e^{-\tau^2}d\tau}\\
&\displaystyle{:=(\bE_1\hat{f})(\xi)+(\bE_2\hat{f})(\xi),\ \xi\in\bR,}
\end{array}
\end{equation}
where
\begin{equation}\label{bE1bE2}
\begin{array}{ll}
&\displaystyle{(\bE_1\hat{f})(\xi):=\hat{f}(\xi)\Big(1-\frac{1}{\sqrt{\pi}}\int_{\frac{(\xi-\pi)r}{\sqrt{2}}}^{\frac{(\xi+\pi)r}{\sqrt{2}}} e^{-\tau^2}d\tau\Big), }\\
&\displaystyle{ (\bE_2\hat{f})(\xi):=\frac{1}{\sqrt{\pi}}\sum_{k\in\bZ\setminus\{0\}}\hat{f}(\xi-2k\pi){\bf 1}_{[-\delta+2k\pi,\delta+2k\pi]}(\xi)\int_{\frac{(\xi-\pi)r}{\sqrt{2}}}^{\frac{(\xi+\pi)r}{\sqrt{2}}} e^{-\tau^2}d\tau.}\\
\end{array}
\end{equation}
By (\ref{f0}) and (\ref{t0}), we get
$$
|(E_1f_0)(t_0)|\ge \|E_1f_0\|_{L^1((0,1))}=\int_{\bR} \big|(E_1f_0)(x){\bf 1}_{(0,1)}(x)\big|dx\ge \Big|\int_{\bR} (E_1f_0)(x){\bf 1}_{(0,1)}(x)dx\Big|.
$$
By the Plancherel theorem (see, e.g.,  Theorem 2.5.6 in \cite{Debnath15}) and (\ref{bE1bE2}), we have
\begin{equation}\label{E1}
\begin{array}{ll}
\displaystyle{|(E_1f_0)(t_0)| }
&\displaystyle{\ge  \Big|\int_{\bR} \widehat{E_1f_0}(\xi)\overline{\widehat{{\bf 1}_{(0,1)}}}(\xi)d\xi \Big| }\\
&\displaystyle{=\Big|\int_{\bR} (\bE_1 \widehat{f_0})(\xi)\widehat{{\bf 1}_{(0,1)}}(-\xi) d\xi+\int_{\bR} (\bE_2\widehat{f_0})(\xi)\widehat{{\bf 1}_{(0,1)}}(-\xi) d\xi \Big| }\\
&\displaystyle{\ge \Big|\int_{-\delta}^{\delta}\widehat{f_0}(\xi)\Big(1-\frac{1}{\sqrt{\pi}}\int_{\frac{(\xi-\pi)r}{\sqrt{2}}}^{\frac{(\xi+\pi)r}{\sqrt{2}}} e^{-\tau^2}d\tau\Big)\widehat{{\bf 1}_{(0,1)}}(-\xi) d\xi}\\
&\displaystyle{\quad +\frac{1}{\sqrt{\pi}}\sum_{k\in\bZ\setminus\{0\}}\int_{-\delta+2k\pi}^{\delta+2k\pi}\widehat{f_0}(\xi-2k\pi)\Big(\int_{\frac{(\xi-\pi)r}{\sqrt{2}}}^{\frac{(\xi+\pi)r}{\sqrt{2}}} e^{-\tau^2}d\tau\Big)\widehat{{\bf 1}_{(0,1)}}(-\xi) d\xi\Big|.}\\
\end{array}
\end{equation}
Note that $\widehat{{\bf 1}_{(0,1)}}(\xi)=\frac{1}{\sqrt{2\pi}} \frac{e^{-i\xi}-1}{-i\xi}$ and $\widehat{f_0}(\xi)=\frac{1}{\sqrt{2\delta}}{\bf 1}_{[-\delta,\delta]}(\xi) e^{-i \xi/2}$, $\xi\in\bR$.
By (\ref{Millseq1}) and (\ref{Millseq2}),  we have
{\small$$
\begin{array}{ll}
&\displaystyle{|(E_1f_0)(t_0)|}\\
&\displaystyle{\ge \frac{1}{2\sqrt{\pi\delta}}\Big| \int_{-\delta}^{\delta}e^{-i \frac{\xi}{2}}\Big(1-\frac{1}{\sqrt{\pi}}\int_{\frac{(\xi-\pi)r}{\sqrt{2}}}^{\frac{(\xi+\pi)r}{\sqrt{2}}} e^{-\tau^2}d\tau\Big)  \frac{e^{i\xi}-1}{i\xi} d\xi  
 +\frac{1}{\sqrt{\pi}}\sum_{k\in\bZ\setminus\{0\}}\int_{-\delta+2k\pi}^{\delta+2k\pi}e^{ik\xi}e^{-i \frac{\xi}{2}}\Big(\int_{\frac{(\xi-\pi)r}{\sqrt{2}}}^{\frac{(\xi+\pi)r}{\sqrt{2}}} e^{-\tau^2}d\tau\Big) \frac{e^{i\xi}-1}{i\xi} d\xi\Big|}\\
&\displaystyle{= \frac{1}{2\sqrt{\pi\delta}}\Big| \int_{-\delta}^{\delta}\Big(1-\frac{1}{\sqrt{\pi}}\int_{\frac{(\xi-\pi)r}{\sqrt{2}}}^{\frac{(\xi+\pi)r}{\sqrt{2}}} e^{-\tau^2}d\tau\Big)\frac{\sin(\xi/2)}{\xi/2} d\xi  
+\frac{1}{\sqrt{\pi}}\sum_{k\in\bZ\setminus\{0\}}(-1)^k\int_{-\delta+2k\pi}^{\delta+2k\pi}\Big(\int_{\frac{(\xi-\pi)r}{\sqrt{2}}}^{\frac{(\xi+\pi)r}{\sqrt{2}}} e^{-\tau^2}d\tau\Big) \frac{\sin(\xi/2)}{\xi/2} d\xi\Big|.  }\\
\end{array}
$$}
Observe that the integrand
$$
\Big(\int_{\frac{(\xi-\pi)r}{\sqrt{2}}}^{\frac{(\xi+\pi)r}{\sqrt{2}}} e^{-\tau^2}d\tau\Big) \frac{\sin(\xi/2)}{\xi/2},\ \xi\in\bR,
$$
is an even function on $\cup_{k\in\bZ\setminus\{0\}}[-\delta+2k\pi,\delta+2k\pi]$, and $(-1)^k=(-1)^{-k}$ for all $k\in\bN$.  Thus,
$$
\begin{array}{ll}
&\displaystyle{|(E_1f_0)(t_0)|}\\
&\displaystyle{\ge \frac{1}{2\sqrt{\pi\delta}}\Big| \int_{-\delta}^{\delta}\Big(1-\frac{1}{\sqrt{\pi}}\int_{\frac{(\xi-\pi)r}{\sqrt{2}}}^{\frac{(\xi+\pi)r}{\sqrt{2}}} e^{-\tau^2}d\tau\Big)\frac{\sin(\xi/2)}{\xi/2} d\xi  
+\frac{4}{\sqrt{\pi}}\sum_{k\in\bN}(-1)^k\int_{-\delta+2k\pi}^{\delta+2k\pi}\Big(\int_{\frac{(\xi-\pi)r}{\sqrt{2}}}^{\frac{(\xi+\pi)r}{\sqrt{2}}} e^{-\tau^2}d\tau\Big) \frac{\sin(\xi/2)}{\xi} d\xi\Big| }\\
&\displaystyle{\ge \frac{1}{2\sqrt{\pi\delta}} \Big[\int_{-\delta}^{\delta}\Big(1-\frac{1}{\sqrt{\pi}}\int_{\frac{(\xi-\pi)r}{\sqrt{2}}}^{\frac{(\xi+\pi)r}{\sqrt{2}}} e^{-\tau^2}d\tau\Big)\frac{\sin(\xi/2)}{\xi/2} d\xi  
-\frac{4}{\sqrt{\pi}}\sum_{k\in\bN} \int_{-\delta+2k\pi}^{\delta+2k\pi}\Big(\int_{\frac{(\xi-\pi)r}{\sqrt{2}}}^{\frac{(\xi+\pi)r}{\sqrt{2}}} e^{-\tau^2}d\tau\Big) \frac{|\sin(\xi/2)|}{\xi} d\xi\Big]. }
\end{array}
$$
Notice that $ \frac{\sin(\delta/2)}{\delta/2}\le \frac{\sin(\xi/2)}{\xi/2}\le 1$ for all $\xi\in[-\delta,\delta]$ 
and  $0\le \frac{|\sin(\xi/2)|}{\xi}\le \frac{\sin(\delta/2)}{2\pi-\delta}$ for all $\xi\in \cup_{k\in\bN}[-\delta+2k\pi,\delta+2k\pi]$. As a result, we compute
$$
\begin{array}{ll}
&\displaystyle{|(E_1f_0)(t_0)|}\\
&\displaystyle{\ge \frac{1}{2\sqrt{\pi\delta}}\Big[ \frac{2\sin(\delta/2)}{\delta} \int_{-\delta}^{\delta}\Big(1-\frac{1}{\sqrt{\pi}}\int_{\frac{(\xi-\pi)r}{\sqrt{2}}}^{\frac{(\xi+\pi)r}{\sqrt{2}}} e^{-\tau^2}d\tau\Big) d\xi  
 -\frac{4\sin(\delta/2)}{\sqrt{\pi}(2\pi-\delta)}\sum_{k\in\bN}\int_{-\delta+2k\pi}^{\delta+2k\pi}\Big(\int_{\frac{(\xi-\pi)r}{\sqrt{2}}}^{\infty} e^{-\tau^2}d\tau\Big)  d\xi\Big]}\\
 &\displaystyle{\ge \frac{\sin(\delta/2)}{\sqrt{\pi\delta}}\Big[ \frac{1}{\delta} \int_{-\delta}^{\delta}\Big(1-\frac{1}{\sqrt{\pi}}\int_{\frac{(\xi-\pi)r}{\sqrt{2}}}^{\frac{(\xi+\pi)r}{\sqrt{2}}} e^{-\tau^2}d\tau\Big) d\xi  
 -\frac{2}{\sqrt{\pi}(2\pi-\delta)}\sum_{k\in\bN}\int_{-\delta+2k\pi}^{\delta+2k\pi}\frac{e^{-\frac{(\xi-\pi)^2r^2}{2}}}{\sqrt{2}(\xi-\pi)r} d\xi\Big]}\\
 &\displaystyle{\ge \frac{\sin(\delta/2)}{\sqrt{\pi\delta}}\Big[ \frac{1}{\delta} \int_{-\delta}^{\delta}\Big(1-\frac{1}{\sqrt{\pi}}\int_{\frac{(\xi-\pi)r}{\sqrt{2}}}^{\frac{(\xi+\pi)r}{\sqrt{2}}} e^{-\tau^2}d\tau\Big) d\xi 
 -\frac{2}{\sqrt{2\pi}(2\pi-\delta)(\pi-\delta)r} \sum_{k\in\bN}\int_{2k\pi-\delta}^{2k\pi+\delta} e^{-\frac{(\xi-\pi)^2r^2}{2}}  d\xi\Big].}\\
\end{array}
$$
By Lemmas \ref{Lemma1} and \ref{lemma2}, we have
$$
\begin{array}{ll}
&\displaystyle{|(E_1f_0)(t_0)|}\\
&\displaystyle{\ge \frac{\sin(\delta/2)}{\sqrt{\pi\delta}}\Big[\frac{1}{\delta}\frac{2\sqrt{2}}{(2+\varepsilon)(\pi+\delta) \sqrt{\pi}}\Big(\frac{2}{(2+\varepsilon)(\pi-\delta)}-\frac{e^{-2\pi\delta r^2}}{\pi+\delta}\Big)\frac{e^{-\frac{(\pi-\delta)^2r^2}{2}}}{r^3} 
-\frac{2}{\sqrt{2\pi}(2\pi-\delta)(\pi-\delta)r}\frac{e^{-\frac{(\pi-\delta)^2r^2}{2}}}{(\pi-\delta)r^2}\Big]}\\
&\displaystyle{\ge \frac{\sin(\delta/2)}{\pi\sqrt{\delta}}\Big[\frac{2\sqrt{2}}{(2+\varepsilon)(\pi+\delta) \delta}\Big(\frac{2}{(2+\varepsilon)(\pi-\delta)}-\frac{e^{-2\pi\delta r^2}}{\pi+\delta}\Big) 
-\frac{2}{\sqrt{2}(2\pi-\delta)(\pi-\delta)^2} \Big] \frac{e^{-\frac{(\pi-\delta)^2r^2}{2}}}{r^3} }\\
&\displaystyle{\ge \frac{\sin(\delta/2)}{\pi\sqrt{2\delta}}\Big[\frac{4}{(2+\varepsilon)(\pi+\delta) \delta}\Big(\frac{2}{(2+\varepsilon)(\pi-\delta)}-\frac{e^{-2\pi\delta r^2}}{\pi+\delta}\Big) 
-\frac{2}{(2\pi-\delta)(\pi-\delta)^2} \Big] \frac{e^{-\frac{(\pi-\delta)^2r^2}{2}}}{r^3},}\\
\end{array}
$$
which completes the proof.
\end{proof}

\vspace{1cm}

{\bf Proof of Theorem \ref{Theorem}:} By (\ref{f0}),  for any $n\ge2$ and $t\in(0,1)$, we estimate
$$
|(E_2f_0)(t)|\le \frac{1}{\sqrt{\pi\delta}}\sum_{j\notin(-n,n]} \Big|\frac{\sin((j-\frac12)\delta)}{(j-\frac12)}  \sinc(t-j) \Big| e^{-\frac{(t-j)^2}{2r^2}}
< \frac{1}{\pi n(n+\frac12)\sqrt{\pi\delta}}\sum_{j\notin(-n,n]}  e^{-\frac{(t-j)^2}{2r^2}}.
$$
By (\ref{Millseq1}),  for any $n\ge2$ and $t\in(0,1)$, we arrive at
$$
\sum_{j\notin(-n,n]}  e^{-\frac{(t-j)^2}{2r^2}}\le 2\sum_{j=n}^{\infty}e^{-\frac{j^2}{2r^2}}=2\sqrt{2}r\sum_{j=n}^{\infty}e^{-\frac{j^2}{2r^2}}\frac{1}{\sqrt{2}r} 
<2\sqrt{2}r\int_{\frac{n-1}{\sqrt{2}r}}^{\infty}e^{-\tau^2}d\tau <\frac{2r^2}{n-1}e^{-\frac{(n-1)^2}{2r^2}}.
$$
Thus,
\begin{equation}\label{E2}
|(E_2f_0)(t_0)|<\frac{2r^2}{\pi n(n+\frac12) (n-1)\sqrt{\pi\delta}}e^{-\frac{(n-1)^2}{2r^2}}, \ n\ge2.  
\end{equation}
Combining (\ref{E1E2}), (\ref{f0}), (\ref{t0}), (\ref{E2}), and Theorem \ref{Theorem1}, it follows easily that
$$
\begin{array}{ll}
\displaystyle{\sup_{\|f\|_{\cB_\delta}\le 1}\sup_{t\in(0,1)} |f(t)-S_{n,r}f(t)|}
&\displaystyle{\ge |(E_1f_0)(t_0)|- |(E_2f_0)(t_0)| }\\
&\displaystyle{\ge \frac{C_{r,\delta,\varepsilon}}{\pi\sqrt{2\delta}} \frac{e^{-\frac{(\pi-\delta)^2r^2}{2}}}{r^3}-\frac{2r^2 e^{-\frac{(n-1)^2}{2r^2}}}{\pi n(n+\frac12) (n-1)\sqrt{\pi\delta}}, }\\
\end{array}
$$
where $C_{r,\delta,\varepsilon}$ is given by (\ref{Crdelta}). The proof is complete. \hspace{6.5cm} $\square$

\vspace{1cm}

At the end of this section, we turn to proving Theorem \ref{Theorem2}.  \\

{\bf Proof of Theorem \ref{Theorem2}: } Let $f\in\cB_\delta$ and $\|f\|_{\cB_\delta}\le 1$. Using (\ref{bE1bE2}) and the Cauchy-Schwartz inequality, we have
$$
\begin{array}{ll}
\displaystyle{\frac{1}{\sqrt{2\pi}}\|\bE_2\hat{f}\|_{L^1(\bR)}}
&\displaystyle{\le \frac{1}{\sqrt{2}\pi} \sum_{0\ne k\in\bZ} \int_{-\delta+2k\pi}^{\delta+2k\pi}|\hat{f}(\xi-2k\pi)|\Big(\int_{\frac{(\xi-\pi)r}{\sqrt{2}}}^{\frac{(\xi+\pi)r}{\sqrt{2}}} e^{-\tau^2}d\tau \Big) d\xi }\\
&\displaystyle{\le \frac{\sqrt{2}}{\pi} \sum_{k\in\bN} \Big[\int_{-\delta+2k\pi}^{\delta+2k\pi} \Big( \int_{\frac{(\xi-\pi)r}{\sqrt{2}}}^{\frac{(\xi+\pi)r}{\sqrt{2}}} e^{-\tau^2}d\tau\Big)^2 d\xi\Big]^{1/2} }\\
&\displaystyle{\le \frac{\sqrt{2}}{\pi} \sum_{k\in\bN} \Big[\int_{-\delta+2k\pi}^{\delta+2k\pi} \Big( \int_{\frac{(\xi-\pi)r}{\sqrt{2}}}^{\infty} e^{-\tau^2}d\tau\Big)^2 d\xi\Big]^{1/2}. }
\end{array}
$$
By (\ref{Millseq1}), we obtain
$$
\begin{array}{ll}
\displaystyle{\frac{1}{\sqrt{2\pi}}\|\bE_2\hat{f}\|_{L^1(\bR)}}
&\displaystyle{\le \frac{\sqrt{2}}{\pi} \sum_{k\in\bN} \Big[\int_{-\delta+2k\pi}^{\delta+2k\pi} \Big(\frac{e^{-\frac{(\xi-\pi)^2r^2}{2}}}{\sqrt{2}(\xi-\pi)r}\Big)^2 d\xi\Big]^{1/2}}\\
&\displaystyle{\le \frac{1}{\pi(\pi-\delta)r} \sum_{k\in\bN} \Big[\int_{-\delta+2k\pi}^{\delta+2k\pi} e^{-(\xi-\pi)^2r^2} d\xi\Big]^{1/2}}\\
&\displaystyle{= \frac{1}{\pi(\pi-\delta)r} \sum_{k\in\bN} \Big[\frac{1}{r} \int_{[(2k-1)\pi-\delta]r}^{[(2k-1)\pi+\delta]r} e^{-\tau^2} d\tau\Big]^{1/2}}\\
&\displaystyle{\le \frac{1}{\pi(\pi-\delta)r} \sum_{k\in\bN} \Big[\frac{1}{r} \int_{[(2k-1)\pi-\delta]r}^{\infty} e^{-\tau^2} d\tau\Big]^{1/2}}\\
&\displaystyle{\le \frac{1}{\pi(\pi-\delta)r} \sum_{k\in\bN} \Big[\frac{1}{r}\frac{1}{2[(2k-1)\pi-\delta]r}e^{-[(2k-1)\pi-\delta]^2r^2}\Big]^{1/2}}\\
&\displaystyle{\le \frac{1}{\pi(\pi-\delta)\sqrt{2(\pi-\delta)}r^2} \sum_{k\in\bN} e^{-[(2k-1)\pi-\delta]^2r^2/2}.}\\
\end{array}
$$
By (\ref{Millseq1}), we compute
$$
\sum_{k=3}^{\infty}e^{-\frac{[(2k-1)\pi-\delta]^2r^2}{2}}=\frac{1}{\sqrt{2}\pi r}\sum_{k=3}^{\infty}e^{-\frac{[(2k-1)\pi-\delta]^2r^2}{2}}\sqrt{2}\pi r\le \frac{1}{\sqrt{2}\pi r}\int_{\frac{(3\pi-\delta)r}{\sqrt{2}}}^{\infty} e^{-\tau^2}d\tau
\le \frac{e^{-\frac{(3\pi-\delta)^2r^2}{2}}}{2\pi(3\pi-\delta)r^2}.
$$
It follows that
$$
\begin{array}{ll}
\displaystyle{\sum_{k\in\bN} e^{-\frac{[(2k-1)\pi-\delta]^2r^2}{2}}  }
&\displaystyle{=e^{-\frac{(\pi-\delta)^2r^2}{2}}+e^{-\frac{(3\pi-\delta)^2r^2}{2}}+\sum_{k=3}^{\infty}e^{-\frac{[(2k-1)\pi-\delta]^2r^2}{2}}  }\\
&\displaystyle{\le  \Big[1+\big(1+\frac{1}{2\pi(3\pi-\delta)r^2}\big)e^{-2\pi(2\pi-\delta)r^2}\Big]  e^{-\frac{(\pi-\delta)^2r^2}{2}}. }\\
\end{array}
$$
Thus, we obtain
\begin{equation}\label{upperboundeq1}
\frac{1}{\sqrt{2\pi}}\|\bE_2\hat{f}\|_{L^1(\bR)}\le \frac{1+\big(1+\frac{1}{2\pi(3\pi-\delta)r^2}\big)e^{-2\pi(2\pi-\delta)r^2} }{\pi(\pi-\delta)\sqrt{2(\pi-\delta)}r^2}e^{-\frac{(\pi-\delta)^2r^2}{2}}. 
\end{equation}
By (2.12) in \cite{LinZhang17}, this leads to
\begin{equation}\label{upperboundeq2}
\frac{1}{\sqrt{2\pi}}\|\bE_1\hat{f}\|_{L^1(\bR)}\le \frac{\sqrt{2\delta} }{\pi(\pi-\delta)r}e^{-\frac{(\pi-\delta)^2r^2}{2}}. 
\end{equation}
By (2.13) in \cite{LinZhang17}, we can assert that
\begin{equation}\label{upperboundeq3}
|E_2(t)|\le \frac{r e^{-\frac{(n-1)^2}{2r^2}}}{\pi n\sqrt{n-1}},\ t\in(0,1). 
\end{equation}
Applying the Riemann-Lebesgue lemma, we get
$$
 |(E_1f)(t)+(E_2f)(t)|\le |(E_1f)(t)|+|(E_2f)(t)|\le  \frac{1}{\sqrt{2\pi}}\|\widehat{E_1f}\|_{L^1(\bR)}+|(E_2f)(t)|,\ t\in(0,1).
$$
Combining (\ref{bE1bE2}),  (\ref{upperboundeq1}),  (\ref{upperboundeq2}), and (\ref{upperboundeq3}), it follows immediately that for each $t\in(0,1)$
$$
\begin{array}{ll}
\displaystyle{ |(E_1f)(t)+(E_2f)(t)| }
&\displaystyle{\le  \frac{1}{\sqrt{2\pi}}\|\bE_1\hat{f}+\bE_2\hat{f}\|_{L^1(\bR)} +|(E_2f)(t)| }\\
&\displaystyle{\le \frac{1}{\sqrt{2\pi}}\|\bE_1\hat{f}\|_{L^1(\bR)}+\frac{1}{\sqrt{2\pi}}\|\bE_2\hat{f}\|_{L^1(\bR)}+|(E_2f)(t)| }\\
&\displaystyle{\le \Big[\frac{1+\big(1+\frac{1}{2\pi(3\pi-\delta)r^2}\big)e^{-2\pi(2\pi-\delta)r^2}}{\sqrt{2(\pi-\delta)}r}+\sqrt{2\delta}\Big]\frac{e^{-\frac{(\pi-\delta)^2r^2}{2}}}{\pi (\pi-\delta)r} +\frac{r e^{-\frac{(n-1)^2}{2r^2}}}{\pi n\sqrt{n-1}}. }\\
\end{array}
$$
The proof is hence complete. \hspace{11cm} $\square$

\section{Numerical experiments}

We shall present numerical experiments to demonstrate that the lower error bound estimation in Theorem \ref{Theorem} for the truncated Gaussian regularized Shannon sampling series (\ref{GR}) is optimal.  To this end, we need to recall the lower error bound (\ref{remarkeq1}) and the upper error bound (\ref{remarkeq2}).

In what follows, we let
$\delta=\frac{\pi}{4}$, $\varepsilon=\frac{1}{7}$, $n\ge 7$ and $r=\sqrt{\frac{n-1}{\pi-\delta}}$. 
Remember that $C_{r,\delta,\varepsilon}$ and $f_0$  are given by (\ref{Crdelta}) and (\ref{f0}), respectively.
An easy computation shows $C_{r,\delta,\varepsilon}\ge 0.0666687$ for all $n\ge7$. We can easily verify that $n\ge 7$ satisfies the condition (\ref{n}).

We shall reconstruct the function values of $f_0$ on $(0,1)$ from $\{f_0(j):j=-n+1,-n+2,\dots, n\}$ by the truncated Gaussian regularized Shannon sampling series
$$
(S_{n,r} f_0)(t)=\sum_{j=-n+1}^n f_0(j)\sinc(t-j)e^{-\frac{(t-j)^2(\pi-\delta)}{2(n-1)}},\ t\in(0,1).
$$
The error of reconstruction is measured by
$$
E(f_0-S_{n,r}f_0):=\max_{1\le j\le 99}\big| (f_0-S_{n,r}f_0)\big(\frac{j}{100}\big)\big|.
$$
This reconstruction error is to be compared with both the lower error bound (\ref{remarkeq1})  denoted by 
$$
E_{\delta,n}:=\Big[0.0666687(\pi-\delta)\sqrt{\pi-\delta}-\frac{2\sqrt{2}}{(\pi-\delta)\sqrt{\pi}} \frac{(n-1)\sqrt{n-1}}{n(n+\frac12)}\Big]
\frac{e^{-\frac{(\pi-\delta)(n-1)}{2}}}{\pi\sqrt{2\delta}(n-1)\sqrt{n-1}},\ n\ge7
$$ 
and the upper error bound (\ref{remarkeq2}) denoted by
$$
{\bf E}_{\delta,n}:=\Big(\sqrt{2\delta}+\frac{\sqrt{n-1}}{n}+\frac{1+(1+\frac{1}{6\pi})e^{-4\pi}}{\sqrt{2(n-1)}}\Big)\frac{e^{-\frac{(\pi-\delta)(n-1)}{2}}}{\pi\sqrt{(\pi-\delta)(n-1)}}, n\ge2.
$$
The above three errors, namely the lower error bound $E_{\delta,n}$, the reconstruction error $E(f_0-S_{n,r}f_0)$  and the upper error bound ${\bf E}_{\delta,n}$, for $n=7,9,\dots,23,25$ are listed in Table \ref{Tab1}. We also plot $\log E_{\delta,n}$, $\log E(f_0-S_{n,r}f_0)$, and $\log {\bf E}_{\delta,n}$ for $n=7,8,\dots,24,25$ in Figure \ref{fig1}. 

To sum up, we proved that the truncated Gaussian regularized Shannon sampling formula converges exponentially rapidly to the bandlimited function and our lower error bound is optimal.

\begin{table}[htbp]
\centering
\begin{tabular}{|c|c|c|c|} \hline
& $E_{\delta,n}$ &$E(f_0-S_{n,r}f_0)$ &${\bf E}_{\delta,n}$  \\ \hline
 n=7 & 7.5816e-07 & 1.6125e-05 &1.3637e-04  \\  \hline
 n=9 & 5.6056e-08 & 1.0218e-06 & 1.0754e-05 \\  \hline
 n=11 & 4.4118e-09 & 7.1272e-08 &8.8497e-07  \\  \hline
 n=13 & 3.5746e-10 & 5.2752e-09 & 7.4813e-08 \\  \hline
 n=15 & 2.9493e-11 &4.0037e-10  &6.4423e-09  \\  \hline
 n=17 & 2.4661e-12 & 3.1085e-11 &5.6227e-10  \\  \hline
  n=19 &2.0841e-13  &2.4961e-12  &4.9577e-11  \\  \hline
 n=21 &1.7768e-14  &2.0497e-13  &4.4065e-12  \\  \hline
 n=23& 1.5261e-15 & 1.6963e-14 & 3.9420e-13 \\  \hline
 n=25 & 1.3193e-16 &1.4843e-15  & 3.5451e-14 \\  \hline
\end{tabular}
\caption{Three errors when $\delta=\frac{\pi}{4}$.} \label{Tab1}
\end{table}

\begin{figure}[htbp]
\centering
  \includegraphics[width=7in]{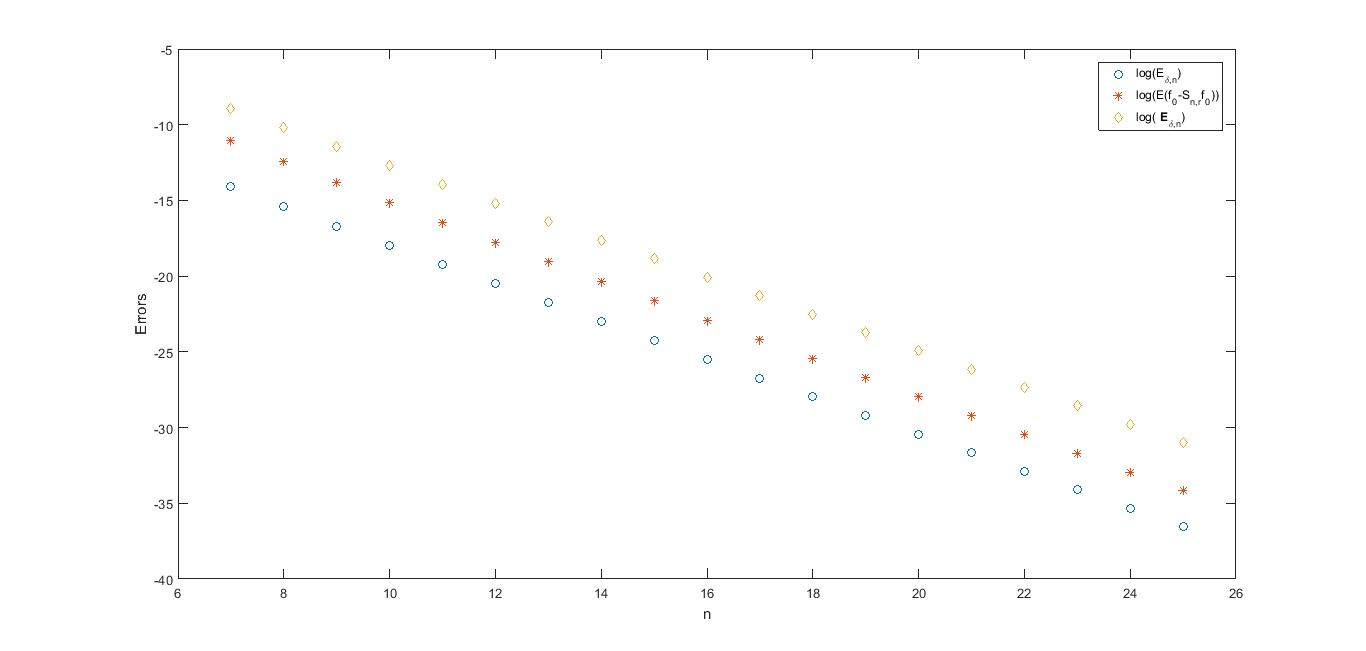}\\
  \caption{Comparison of $\log (E_{\delta,n})$, $\log E(f_0-S_{n,r}f_0)$, and $\log({\bf E}_{\delta,n})$ when $\delta=\frac{\pi}{4}$.}\label{fig1}
\end{figure}

{\small
\bibliographystyle{amsplain}

\begin{thebibliography}{30}
\bibitem{Aldroubi} A. Aldroubi (2002). Non-uniform weighted average sampling and reconstruction in shift-invariant and wavelet spaces. {\it Appl. Comput. Harmon. Anal.}  13:151--161.
\bibitem{Chen15} W. Chen and H. Zhang (2015). Exponential approximation of bandlimited random processes from oversampling. {\it Sci. Sin. Math.} 45:167--182.
\bibitem{Chen1} W. Chen and H. Zhang. Exponential approximation of multivariate bandlimited functions from average oversampling.  arxiv: 1412.4265v1.
\bibitem{Grochenig} K. Gr\"{o}chenig (1992). Reconstruction algorithms in irregular sampling. {\it Math. Comput.} 59:181--194.
\bibitem{Debnath15} L. Debnath and D. Bhatta (2015). {\it Integral Transforms and Their Applications}. 3rd ed., Boca Raton, FL,  USA: CRC Press, pp. 18--44.
\bibitem{Helms} H. D. Helms and J. B. Thomas (1962). Truncation error of sampling-theorem expansions. {\it Proc. IRE} 50:179--184.
\bibitem{Jagerman} D. Jagerman (1966). Bounds for truncation error of the sampling expansion. {\it SIAM J. Appl. Math.}  14:714--723.
\bibitem{Jerri} A. J. Jerri (1977). The Shannon sampling theorem--Its various extensions and applications: A tutorial review.  {\it Proc. IEEE} 65:1565--1596.
\bibitem{Kotelnikov33} V. A. Kotel'nikov (1933). On the carrying capacity of the ether and wire in telecommunications, Material for the First All-Union Conference on Questions of Communication. {\it Izd. Red. Upr. Svyazi RKKA}, Moscow.
\bibitem{LinZhang17} R. Lin and H. Zhang (2017). Convergence analysis of the Gaussian regularized Shannon sampling series. {\it Numer. Funct. Anal. Optim.}  38:224--247.
\bibitem{Marks} R. J. Marks II (1991).  {\it Introduction to Shannon Sampling and Interpolation Theory}. Springer Texts in Electrical Engineering, New York, NY, USA: Springer-Verlag, pp. 56--102.
\bibitem{Micchelli} C. A. Micchelli, Y. Xu, and H. Zhang (2009). Optimal learning of bandlimited functions from localized sampling. {\it J. Complexity}  25:85--114.
\bibitem{Nyquist28} H. Nyquist (1928). Certain topics in telegraph transmission theory.  {\it Trans. AIEE} 47: 617--644.
\bibitem{Pollak} H. D. Pollak (1956). A remark on `elementary inequalities for Mills' ratio' by Y\^{u}saku Komatu. {\it Rep. Stat. Appl. Res., Un. Jap. Sci. Engrs.}, 4:110.
\bibitem{Qian} L. Qian (2003). On the regularized Whittaker-Kotel'nikov-Shannnon sampling formula. {\it Proc. Amer. Math. Soc.} 131:1169--1176.
\bibitem{Qianthesis} L. Qian (2004). The regularized Whittaker-Kotel'nikov-Shannon sampling theorem and its application to the numerical solutions of partial differential equations. Ph.D. dissertation, Dept. Comput. Sci., Natl. Univ. Singapore, Singapore.
\bibitem{Shannon} C. E. Shannon (1949). Communication in the presence of noise. {\it Proc. IRE} 37:10--21.
\bibitem{SongSun07} Z. Song, W. Sun, X. Zhou, and Z. Hou (2007). An average sampling theorem for bandlimited stochastic processes. {\it IEEE Trans. Inform. Theory}  53:4798--4800.
\bibitem{SunZhou02} W. Sun and X. Zhou (2002). Reconstruction of band-limited signals from local averages. {\it IEEE Trans. Inform. Theory}  48:2955--2963.
\bibitem{Unser} M. Unser (2000). Sampling--50 years after Shannon.  {\it Proc. IEEE} 88:569--587.
\bibitem{Vaidyanathan} P. P. Vaidyanathan (2001). Generalizations of the sampling theorem: seven decades after Nyquist. {\it IEEE Trans. Circuits Systems I Fund. Theory Appl.} 48:1094--1109. 
\bibitem{Wei98} G. W. Wei (1998). Quasi wavelets and quasi interpolating wavelets. {\it Chem. Phys. Lett.} 296:215--222.
\bibitem{Wei00} G. W. Wei (2000). Discrete singular convolution method for the Sine-Gordon equation. {\it Physica D} 137:247--259.
\bibitem{Wei000} G. W. Wei (2000). Solving quantum eigenvalue problems by discrete singular convolution. {\it J. Phys. B: At. Mol. Opt. Phys.} 33:343--359.
\bibitem{Wei08} G. W. Wei and S. Zhao (2007). On  the validity of  `A proof  that  the discrete  singular  convolution (DSC)/Langrange-distributed  approximation  function  (LDAF)  method is  inferior to  high order finite  differences'.  {\it J. Comput. Phys.} 226:2389--2392.
\bibitem{Whittaker} E. T. Whittaker (1915).  On the functions which are represented by the expansion of the interpolation theory. {\it Proc. Roy. Soc. Edinburgh Sect. A}  35:181--194.
\bibitem{Zhao} S. Zhao and G. W. Wei (2003). Comparison of the discrete singular convolution and three other numerical schemes for solving Fisher's equation. {\it SIAM J. Sci. Comput.}  25:127--147.

\end{thebibliography}

}

\end{document}